\documentclass[10pt]{article}
\font\smallit=cmti10
\font\smalltt=cmtt10

\usepackage{amssymb,latexsym,amsmath,epsfig,amsthm} 

\makeatletter

\renewcommand\section{\@startsection {section}{1}{\z@}
{-30pt \@plus -1ex \@minus -.2ex}
{2.3ex \@plus.2ex}
{\normalfont\normalsize\bfseries}}

\renewcommand\subsection{\@startsection{subsection}{2}{\z@}
{-3.25ex\@plus -1ex \@minus -.2ex}
{1.5ex \@plus .2ex}
{\normalfont\normalsize\bfseries}}

\renewcommand{\@seccntformat}[1]{\csname the#1\endcsname. }

\renewcommand{\(}{\ensuremath{\left(}}
\renewcommand{\)}{\ensuremath{\right)}}
\renewcommand{\[}{\ensuremath{\left[}}
\renewcommand{\]}{\ensuremath{\right]}}

\makeatother

\newtheorem{theorem}{Theorem}

\newtheorem{conjecture}{Conjecture}

\theoremstyle{plain}
\newtheorem{problem}[theorem]{Problem}

\usepackage{algorithm}
\usepackage{algpseudocode}

\usepackage[colorlinks=true,citecolor=black,linkcolor=black,urlcolor=blue]{hyperref}

\begin{document}

\begin{center}
\uppercase{\bf A Note on 3-free Permutations}
\vskip 20pt
{\bf Bill Correll, Jr.}\\
{\smallit MDA Information Systems LLC, Ann Arbor, MI, USA}\\
{\tt william.correll@mdaus.com}\\
\vskip 10pt
{\bf Randy W. Ho}\\
{\smallit Garmin International, Chandler, AZ, USA}\\
{\tt GooglyPower@gmail.com}\\
\end{center}
\vskip 30pt
\centerline{\smallit Received: , Revised: , Accepted: , Published: } 
\vskip 30pt

\centerline{\bf Abstract}

\noindent
Let $\theta(n)$ denote the number of permutations of $\{1,2,\ldots,n\}$ that do not contain a 3-term arithmetic progression as a subsequence.  Such permutations are known as 3-free permutations.  We present a dynamic programming algorithm to count all 3-free permutations of $\{1,2,\ldots,n\}$.  We use the output to extend and correct enumerative results in the literature for $\theta(n)$ from $n=20$ out to $n=90$ and use the new values to inductively improve existing bounds on $\theta(n)$. 


\pagestyle{myheadings} 
\markright{\smalltt INTEGERS: 17 (2017)\hfill} 
\thispagestyle{empty} 
\baselineskip=12.875pt 
\vskip 30pt

\section{Introduction and Results}
\label{intro}
Let $n$ be a positive integer and let $\[n\]$ denote the set $\{1,2,\ldots,n\}$.  Let $\alpha=\(a_1, a_2, \ldots, a_n\)$ be a permutation of $\[n\]$.  Then $\alpha$ is a \textit{3-free permutation} if and only if, for every index $j$ $(1 \leq j \leq n)$, there do not exist indices $i < j$ and $k > j$ such that $a_i + a_k = 2a_j$.  Let $\theta(n)$ be the function that gives the number of 3-free permutations of $\[n\]$.  Of course the value of $\theta(n)$ will be unchanged if we replace $\[n\]$ with any set of $n$ integers in arithmetic progression so we will hereafter use $\[n\]$ when referring to $\theta(n)$.  In 1973 Entringer and Jackson initiated the study of 3-free permutations by posing

\begin{problem}
[Entringer and Jackson \cite{E2440source}]
Does every permutation of $\{0, 1, \ldots, n\}$ contain an arithmetic progression of at least three terms?
\label{keyprob}
\end{problem}

Three solutions (see \cite{lyndonsol}, \cite{oddasol}, \cite{thomassol}) to Problem \ref{keyprob} showing that the answer is ``No'' along with comments \cite{simmons} containing a table of values of $\theta(n)$ for $1 \leq n \leq 20$ were published.  The solutions of Odda \cite{oddasol} and Thomas \cite{thomassol} contained the first constructions for 3-free permutations.  Odda describes how to construct one 3-free permutation for each $n$.  Thomas devised a method to generate $2^{n-1}$ 3-free permutations for each $n$.  Thomas's examples show that the sets of permutations his method generates aren't exhaustive.  

The purpose of this note is to present an algorithm that counts the number of $3$-free permutations of $n$ consecutive integers for each $n$.  We correct and extend the tables of known values of $\theta(n)$ out to $n=90$ and improve upper and lower bounds by proving the following four results.

\begin{theorem}
\label{mainthm2}
For positive integers $n \geq 45$,
\begin{eqnarray}
\theta(n) \geq \frac{c_1^n}{2}, c_1 = \sqrt[80]{2\theta(80)} = 2.201\ldots.
\end{eqnarray}
\end{theorem}

\begin{theorem}
\label{mainthm1}
For positive integers $n \geq 36$,
\begin{eqnarray}
\theta(n) \leq \frac{c_2^n}{21}, c_2 = \sqrt[64]{21\theta(64)} = 2.364\ldots.
\end{eqnarray}
\end{theorem}

\begin{theorem}
\label{mainthm3}
For positive integers $k \geq 6$ and $n=2^k$,
\begin{eqnarray}
\theta(n) \geq \frac{c_3^n}{2}, c_3 = \sqrt[64]{2\theta(64)} = 2.279\ldots.
\end{eqnarray}
\end{theorem}

\begin{theorem}
\label{mainthm4}
For all positive integers $n$,
\begin{eqnarray}
\theta(n) \geq \frac{nc_4^n}{40}, c_4 = \sqrt[40]{\theta(40)} = 2.156\ldots.
\end{eqnarray}
\end{theorem}

The existence of $\lim \theta(n)^{1/n}$ as $n \rightarrow \infty$ was identified in \cite{sharma} as a key problem in the study of $\theta(n)$.  It remains an open question although Theorems \ref{mainthm2}, \ref{mainthm1} imply that the limit lies within the interval $\[c_1 , c_2\]$ if it exists.  The first author explored connections between 3-free permutations and Costas arrays in \cite{correll}, where slightly weaker versions of Theorems \ref{mainthm2} and \ref{mainthm1} were stated without proof.  

For clarity, we comment here that we are not presenting any results on the related problem of evaluating and bounding the function $r(n)$ giving the longest 3-free subsequence of the sequence $1,2,\ldots,n$.  The latest developments in solving that problem currently appear in \cite{Latestonr}. 

\section{Some Results From the Literature on $\theta(n)$}
Davis, Entringer, Graham, and Simmons \cite{degs} established a number of bounds on the growth of $\theta(n)$ including the following:


\begin{theorem}
[Davis, Entringer, Graham, and Simmons, \cite{degs}]
For positive integers $n$,
\label{inductivethm2}
\begin{align}
\theta(2n) & \geq 2\theta^2(n), \\
\theta(2n+1) & \geq 2\theta(n)\theta(n+1).
\end{align}
\end{theorem}  

\begin{theorem}
[Davis, Entringer, Graham, and Simmons, \cite{degs}]
\label{DEGSThm}
For $n=2^k, k \geq 4$,
\begin{eqnarray}
\theta(n)~\geq~\frac{c^n}{2}, c=\sqrt[16]{2\theta(16)}=2.248\ldots.
\label{lpow2}
\end{eqnarray}
\end{theorem}

Sharma's dissertation \cite{sharmadiss} is noteworthy in that it established the long-conjectured result that $\theta(n)$ has an exponential upper bound.  Sharma used parity arguments to prove 

\begin{theorem}
[Sharma, \cite{sharma}]
For each $n \geq 3$,
\begin{eqnarray}
\theta(n) \leq 21\theta\(\left\lceil\frac{n}{2}\right\rceil\)\theta\(\left\lfloor\frac{n}{2}\right\rfloor\).
\end{eqnarray}
\label{sharmaindthm}
\end{theorem}    
The key result in \cite{sharmadiss} (and in the follow-up journal paper \cite{sharma} as well as the book \cite{sharma2011structure}) he obtains from Theorem \ref{sharmaindthm} is 

\begin{theorem}
[Sharma, \cite{sharma}]
\label{SThm}
For $n \geq 11$,
\begin{eqnarray}
\label{sharmamain}
\theta(n)~\leq~\frac{2.7^n}{21}.
\end{eqnarray}
\end{theorem}

Sharma also improved Thomas's strict, constructive lower bound of $2^{n-1}$ for $n > 5$ by showing that: 

\begin{theorem}
[Sharma, \cite{sharma}]
\label{SharmaLowAll}
For all positive integers $n$,
\begin{eqnarray}
\theta(n)~\geq~\frac{n2^n}{10},
\end{eqnarray}
\end{theorem}
but LeSaulnier and Vijay were able to establish

\begin{theorem}
[LeSaulnier and Vijay \cite{LeSV}]
For $n \geq 8$,
\label{LSVThm}
\begin{eqnarray}
\theta(n)~\geq~\frac{1}{2}c^n, \mbox{ where } c~=~(2\theta(10))^{\frac{1}{10}}~=~2.152.... 
\label{LSV} 
\end{eqnarray}
\end{theorem}

In Section \ref{betterbound} we use Theorems \ref{inductivethm2} and \ref{sharmaindthm} in inductive proofs of Theorems \ref{mainthm2}, \ref{mainthm1}, and \ref{mainthm3} to improve Theorems \ref{LSVThm}, \ref{SThm}, and \ref{DEGSThm}, respectively.  We also rework Sharma's proof of Corollary 3.2.1 of \cite{sharma} relying on Theorem \ref{inductivethm2} using our additional computed values of $\theta(n)$ to improve upon Theorem \ref{SharmaLowAll} for $n \geq 19$.

\section{Algorithm Descriptions}
\label{algdes}

Recall from Section \ref{intro} that we write $\alpha=\(a_1, a_2, \ldots, a_n\)$ for a permutation of $\[n\]$.
For $j=1, 2, \ldots, n$, if we define 
\begin{eqnarray}
T_j \equiv \{ a_k \, | \, j \leq k \leq n \},
\end{eqnarray}
then the 3-free property of a permutation can be restated as saying that there do not exist $a_i \not \in T_j$ and $a_k~\in T_j - \{a_j\}$ such that $a_i + a_k = 2a_j$.  Further inspection of the 3-free property allows us to replace $T_j - \{a_j\}$ with $T_j$, because  $a_k = a_j$ would imply $a_i=a_j$ 
(from $a_i + a_k=2a_j$), but then $a_i$ would be in $T_j$. 
If the 3-free property holds for all $1\le j \le n$, then $\alpha$ is a 3-free permutation.
This suggests the following algorithm to generate 3-free permutations:

\begin{algorithm}\label{alg1}
\textbf{Backtracking algorithm to enumerate 3-free permutations}

\textbf{Subroutine} Enumerate

\textbf{Input:}  A (possibly empty) sequence $\rho=(p_1,p_2,\ldots,p_k)$ of distinct integers $p_k \in [n]$

\textbf{Output:}  All 3-free permutations of $[n]$ that begin with $\rho$

\begin{algorithmic}[1]
\Procedure{Enumerate}{$\rho$}
\If {$|\rho| = n$}
    \State {print $\rho$}  
\Else
    \State  $P=\{\rho_1,\rho_2,\ldots,\rho_{|\rho|}\}$
    \For{$1 \le j \le n$}
        \If {$j \notin P$ and $\nexists i\in P, k\in \[n\] \backslash P$ such that $i+k=2j$}
            \State \Call{Enumerate}{$(\rho, {j})$}
        \EndIf
    \EndFor
\EndIf
\EndProcedure
\end{algorithmic}

\textbf{Main Backtracking Algorithm}

\textbf{Input:}  Positive integer $n$

\textbf{Output:}  List of all 3-free permutations of $\[n\]$

\begin{algorithmic}[1]
\Procedure{EnumerateMain}{$n$} 
\State { $\Call{Enumerate}{[ ]} \Comment{ \mbox {Empty sequence} }$}
\EndProcedure
\end{algorithmic}

\end{algorithm}

In the subroutine \textit{Enumerate}, the notation $(\rho,j)$ on line 8 denotes the sequence obtained from appending the integer $j$ to the sequence $\rho$.  The main backtracking algorithm recursively generates the 3-free permutations one by one, so it could be of use in generating data from which new structural properties of 3-free permutations could be deduced.  Clearly its running time is bounded below by $\theta(n)$.  If we are interested in the number $\theta(n)$ of 3-free permutations and not the permutations themselves, we can speed up the counting process by using dynamic programming.  Dynamic programming algorithms (see, for instance, \cite{cormen}) solve programs by combining solutions to subproblems.  The subproblems can be dependent in that they have common subsubproblems.    

A key observation is that the 3-free property of a permutation depends on the {\em set} of elements that have been used so far in building up that permutation. 
The exact ordering of those elements is not relevant.  The dynamic programming algorithm recursively evaluates $\theta(n)$ using dynamic programming.  It uses bitsets to keep track of which integers have not been placed in an effort to build up a 3-free permutation (a \textit{bitset} is a sequence of zeros and ones.)  

\begin{algorithm}\label{alg2}

\textbf{Dynamic programming algorithm to count 3-free permutations}

\textbf{Subroutine} Count

\textbf{Input:}  A bitset $b$ of length $n$   

\textbf{Output:}  The number $\theta(b)$ of 3-free permutations of $[n]$ that begin with $\rho$,
where $\rho$ is any valid initial sequence that uses exactly the integers that $b$ maps to 0. Note that if there is more than one such sequence,
then they must give the same number, due to the 3-free property

\begin{algorithmic}[1]
\Function{Count}{$b$}
\If { $\exists (b,v) \in C$ for some $v$}
    \State \Return  $v$
\ElsIf {$max\{b[1],b[2],\ldots,b[n]\} = 0$}
    \State \Return 1
\Else
    \State {$ans \gets 0$}
    \For {$1 \leq j \leq n$}
        \If {$b[j] = 1$ and $\nexists 1 \le i, k \le n$ such that $b[i] = 0$ and $b[k] = 1$ and $i + k = 2j$}
            \State {$b' \gets b$}
            \State {$b'[j] \gets 0$}
            \State {$ans \gets ans + \Call{Count}{b'}$}
        \EndIf
    \EndFor
    \State {$C \gets C \cup \{  (b, ans) \}$}
    \State \Return ans
\EndIf
\EndFunction
\end{algorithmic}

\textbf{Main Dynamic Programming Algorithm}

\textbf{Input:}  Positive integer $n$

\textbf{Output:}  $\theta(n)$

\begin{algorithmic}[1]
\Function{CountMain}{$n$}  
\State {$C \gets \emptyset$}
\State \Return $\Call{Count}{(1, 1, \ldots, 1)}\Comment{ \mbox{Bitset of $n$ ones} }$
\EndFunction
\end{algorithmic}

\end{algorithm}

In the above algorithm, $C$ denotes a set of pairs $(b,v)$, where $b$ is a bitset of length $n$ and $v$ is a non-negative integer. The set $C$ is intended to be implemented by a data structure known as a ``map''.  In our usage of C, the value of v for each b is $\theta(b)$.

Let $T \subseteq \[n\]$.  It takes $O(n)$ time to check if the 3-free property is violated and it takes $O(n)$ time to iterate over every element $t$ in $T$.  On the surface Algorithm \ref{alg2} appears to require $O(2^n)$ memory to store $\theta(T)$ for every subset $T$ of $\[n\]$ and the running time appears to be $O(n^2 2^n)$.  However, it turns out that only a small percentage of the subsets of $\[n\]$ are needed in the recurrence because most of them are not 
reachable due to a violation of the 3-free property. This helped us to tabulate $\theta(n)$ out to $n=90$.
The value of $\theta(90)$ has 31 digits.

\section{Computational Enumerative Results}
\label{enumres}

We pushed a Java implementation of the dynamic programming algorithm out to $n=90$ and updated entry A003407 of the Online Encyclopedia of Integer Sequences (\url{http://www.oeis.org/A003407}) with the values in Table~\ref{T1}.  For $n=90$, the fraction of subsets that had to be visited was only

\begin{eqnarray}
254931123/(2^{90}) \approx 2.059(10^{-19}).
\end{eqnarray}

Our Java implementation ran out of memory for $n=91$.  Algorithm \ref{alg2} does not lend itself to parallelization due to the way it uses memory.  Additional values of $\theta(n)$ can be obtained on computing platforms having additional memory, support for arbitrarily long integers, and adequate processing power. 

Before our computations, there were at least 4 published tables of values of $\theta(n)$ for $1 \leq n \leq 20$ although only two of these tables are correct.  The very first table to appear is in \cite{simmons} and claims that $73904$ is the value of $\theta(15)$ but the correct value is $\theta(15)=74904$.   For $n=17$ the table in \cite{sharma} claims that $360016$ is the value of $\theta(17)$ but the correct value is $\theta(17)=368016$.  The first twenty values of $\theta(n)$ listed above do agree with the table in \cite{degs}.  The first 20 entries in entry A003407 were correct at the time we extended them.

\begin{table}
\caption{Number of 3-free permutations $\theta(n)$ of $\[n\]$}
\label{T1} 
\begin{center}
\begin{tabular}{|c|c|c|c|c|c|}
\hline
$n$ & $\theta(n)$ & $n$ & $\theta(n)$ & $n$ & $\theta(n)$ \\
\hline \hline
1 & 1 & 31 & 41918682488 & 61 & 1612719155955443585092 \\
\hline
2 & 2 & 32 & 121728075232 & 62 & 4640218386156695178110 \\
\hline
3 & 4 & 33 & 207996053184 & 63 & 13557444070821420327240 \\
\hline
4 & 10 & 34 & 360257593216 & 64 & 39911512393313043466768 \\
\hline
5 & 20 & 35 & 639536491376 & 65 & 67867319248960144994224 \\
\hline
6 & 48 & 36 & 1144978334240 & 66 & 115643050433241064474672 \\
\hline
7 & 104 & 37 & 2362611440576 & 67 & 199272038058617170554928 \\
\hline
8 & 282 & 38 & 4911144118024 & 68 & 344053071167567188894208 \\
\hline
9 & 496 & 39 & 10417809568016 & 69 & 608578303898604406167840 \\
\hline
10 & 1066 & 40 & 22388184630824 & 70 & 1080229099508551381463536 \\
\hline
11 & 2460 & 41 & 50301508651032 & 71 & 1929269192569465070403584 \\
\hline
12 & 6128 & 42 & 113605533519568 & 72 & 3452997322628833453585008 \\
\hline
13 & 12840 & 43 & 265157938869936 & 73 & 7096327095079914521075040 \\
\hline
14 & 29380 & 44 & 622473467900178 & 74 & 14611112240136930804928288 \\
\hline
15 & 74904 & 45 & 1527398824248200 & 75 & 30235147387260979648843264 \\
\hline
16 & 212728 & 46 & 3784420902143392 & 76 & 62757445134327428602306464 \\
\hline
17 & 368016 & 47 & 9503564310606436 & 77 & 132956581436718531491070160 \\
\hline
18 & 659296 & 48 & 23991783779046768 & 78 & 282272593229156186280461264 \\
\hline
19 & 1371056 & 49 & 48820872045382552 & 79 & 605672649054377049472147568 \\
\hline
20 & 2937136 & 50 & 99986771685259808 & 80 & 1302375489530691442230524528 \\
\hline
21 & 6637232 & 51 & 209179575852808848 & 81 & 2914298247043287576460093712 \\
\hline
22 & 15616616 & 52 & 441563057878399888 & 82 & 6537258415569149903366841040 \\
\hline
23 & 38431556 & 53 & 992063519708141728 & 83 & 14713284774210886488265138336 \\
\hline
24 & 96547832 & 54 & 2241540566114243168 & 84 & 33155372641605493828236640928 \\
\hline
25 & 198410168 & 55 & 5185168615770591200 & 85 & 77219028670778815210019118736 \\
\hline
26 & 419141312 & 56 & 12057653703359308256 & 86 & 180104653062631494787580542664 \\
\hline
27 & 941812088 & 57 & 31151270610676979624 & 87 & 421733920870430143234318231648 \\
\hline
28 & 2181990978 & 58 & 81046346414827952010 & 88 & 990082990967384066255452324186 \\
\hline
29 & 5624657008 & 59 & 213208971281274232760 & 89 & 2428249522507620383597702223224 \\
\hline
30 & 14765405996 & 60 & 563767895033816986864 & 90 & 5963505178650560845887322154368 \\
\hline
\end{tabular}
\end{center}
\end{table}

\section{Proofs}
\label{betterbound}
Theorems \ref{mainthm2}, \ref{mainthm1}, and \ref{mainthm3} can be proven by induction:

\begin{proof} To prove Theorem \ref{mainthm2} it suffices, by Theorem \ref{inductivethm2}, to prove $\theta(n)~\geq~\frac{c^n}{2}$ for $42~\leq~n~\leq 83$ and some constant $c$.  Computation shows that the maximal such $c$ is $\min(2\theta(n))^{1/n} = c_1$ and occurs for $n=42$.
\end{proof}  

\begin{proof} To prove Theorem \ref{mainthm1}, we observe that, for $42 \leq n \leq 83, \max(21\theta(n))^{\frac{1}{n}} = c_2,$ and occurs for $n=64$ so (\ref{mainthm1}) holds for all $n \in \[42, 83\]$.  That inequality (\ref{mainthm1}) holds for all $n~\geq~42$ follows from using the fact that it holds for $42 \leq n \leq 83$ as a basis for an inductive argument and from Theorem \ref{sharmaindthm}.  Straightforward numerical investigation reveals that inequality (\ref{mainthm1}) actually holds for $n \geq 36$ (but not for $n \leq 35$).
\end{proof}  

\begin{proof} A proof of Theorem \ref{mainthm3} follows by induction on $k$ using Theorem \ref{inductivethm2}.
\end{proof}

To prove Theorem \ref{mainthm4}, we rework the reasoning of Section 3 of \cite{sharma} through the proof of Corollary 3.2.1 using an exponential base $\alpha > 2$ and the values of $\theta(n)$ in Table~\ref{T1}.  We obtain improved variants of Theorems 3.1 and 3.2 of \cite{sharma} along the way.  First note that:

(a) If for some positive integer $n, \theta(n) \geq \alpha^n$ and $\theta(n+1) \geq \alpha^{n+1}$ then by Theorem \ref{inductivethm2}, we have  $\theta(2n) \geq 2\alpha^{2n}$ and $\theta(2n+1) \geq 2\alpha^{2n+1}$.  

By computer verification and the data in Table~\ref{T1} , we see that $\theta(n) \geq \alpha^n$ for $n \in \left[40,79\right]$ for $\alpha = c_4$ and that $c_4$ is the maximal such value.  Thus, by (a), $\theta(n) \geq 2\alpha^n$ for $n \in \left[80,159\right]$.  Applying (a) to this inequality yields $\theta(n) \geq 8 \alpha^{n}$ for $n \in \left[160,319\right]$.  An inductive argument allows us to prove the following improvement on Theorem 3.1 of \cite{sharma}.

\begin{theorem} For integers $p \geq 2$ and $\alpha = c_4$,
\label{newThm3pt1}
\begin{eqnarray}
\theta(n) \geq 2^{2^{p-2}-1}\alpha^{n} \mbox{ for all } n \in \left[5\times 2^{p+1} , 5\times 2^{p+2}-1\right].
\end{eqnarray}
\end{theorem}

\begin{proof}We know the statement is true for $p=2$.  Suppose the statement holds for all $p \leq l-1$.  Then for $n \in \left[5 \times 2^{l+1} , 5 \times 2^{l+2} - 1 \right]$ if $n$ is even, applying the inductive hypothesis to $\frac{n}{2}$ and using Theorem \ref{inductivethm2} verifies the theorem for $n$ (similarly for $n$ odd applying the induction hypothesis to $\frac{n-1}{2}$ and $\frac{n+1}{2}$).  This verifies the statement for $p=l$ as desired.
\end{proof}

Next, we prove the following improvement over Theorem 3.2 of \cite{sharma}.

\begin{theorem} For any fixed integer $p \geq 5$ and $\alpha = c_4$,
\label{newThm3pt2}
\begin{eqnarray}
\lim_{n \rightarrow \infty} \frac{\theta(n)}{n^p \times \alpha^n} = \infty.
\end{eqnarray}
\end{theorem}

\begin{proof} Consider the sequence $a_n = \frac{\theta(n)}{n^{p+1} \times \alpha^n}$ for $n \geq 5 \times 2^{p+1}$.  Note that $a_{2n} = \frac{\theta(2n)}{(2n)^{p+1} \times \alpha^{2n}} \geq \frac{2 \times \left[\theta(n) \right]^2}{(2n)^{p+1} \times \alpha^{2n}} \geq \frac{\theta(n)}{n^{p+1} \times \alpha^n} \times \frac{\theta(n)}{\alpha^{n+p}} = a_n \times \frac{\theta(n)}{\alpha^{n+p}} \geq a_n$ (as $\theta(n) \geq 2^{2^{p-2}-1}\alpha^n$ and $2^{p-2}-1 \geq p\log_2\alpha$ for the intervals of $n$ and $p$ values).  Similarly $a_{2n+1} \geq a_{n+1}$ for all such $n$ (proof is identical with the additional step of noting that $(2n+2)^{p+1} \geq (2n+1)^{p+1}$).  Let $\gamma = \min a_n$ for $n \in \left[5 \times 2^{p+1} , 5 \times 2^{p+2} - 1\right]$.  Using the statements $a_{2n} \geq a_n$ and $a_{2n+1} \geq a_{n+1}$ recursively implies $a_n \geq \gamma$ for all $n \geq 5 \times 2^{p+1}$.  Therefore $\frac{\theta(n)}{n^p \times \alpha^n} = n \times a_n \geq n \times \gamma$ for all $n \geq 5 \times 2^{p+1}$ and $\frac{\theta(n)}{n^p \times \alpha^n}$ clearly tends to $\infty$ as $n \rightarrow \infty$ as desired.
\end{proof}

We now prove Theorem \ref{mainthm4}.

\begin{proof}Let $a_n =\frac{\theta(n)}{n \times \alpha^n}$.  From the values of $\theta(n)$ in Table~\ref{T1} we note that $a_n \geq \frac{1}{40}$ for all $n \in [40,79]$.  Since $\theta(n) \geq \alpha^n$ for all $n \geq 40$ by Theorem \ref{newThm3pt1}, reasoning as in the proof of Theorem \ref{newThm3pt2} lets us prove that $a_{2n} \geq a_n$ and $a_{2n+1} \geq a_{n+1}$ for all $n \geq 40$.  This proves that $a_n \geq \frac{1}{40}$ for all positive integers $n$.
\end{proof}

\section{Conjecture}

Define the function $h(n) = log(\theta(n+1)) - log(\theta(n))$.  Examining a plot of $h(n)$ suggests

\begin{conjecture}
The function $h(n)$ is increasing on the intervals $\[2^k , 2^k + 2^{k-1} - 1\]$ and $\[2^k + 2^{k-1}, 2^{k+1} - 1\]$ but is decreasing on $\[2^k + 2^{k-1} - 1 , 2^k + 2^{k-1}\]$ and the interval $\[2^{k+1}-1,2^{k+1}\]$ for $k \geq 2$.
\end{conjecture}


\bibliographystyle{plain}

\end{document}